\documentclass[12pt]{amsart}
\usepackage{amssymb,latexsym}
\usepackage{enumerate}

\makeatletter
\@namedef{subjclassname@2010}{%
  \textup{2010} Mathematics Subject Classification}
\makeatother

\newtheorem{theorem}{Theorem}[section]
\newtheorem{cor}[theorem]{Corollary}
\newtheorem{lemma}[theorem]{Lemma}
\newtheorem{prop}[theorem]{Proposition}

\theoremstyle{definition}

\newtheorem{rem}[theorem]{Remark}

\numberwithin{equation}{section}

\newcommand{\rad}{{\rm rad}}

\newcommand{\NN}{{\mathcal N}}

\newcommand{\PP}{{\mathcal P}}
\newcommand{\e}{{\rm e}}
\newcommand{\D}{{\mathcal D}}

\renewcommand{\S}{{\mathcal S}}
\newcommand{\Z}{{\mathbb Z}}

\begin{document}

\title{Product-free sets with high density}

\author[P. Kurlberg]{P\"ar Kurlberg}
\address{Department of Mathematics\\
KTH\\
SE-10044, Stockholm, Sweden}
\email{kurlberg@math.kth.se}

\author[J. C. Lagarias] {Jeffrey C.  Lagarias}
\address{Department of Mathematics\\
University of Michigan\\
Ann Arbor, MI 48109, USA}
\email{lagarias@umich.edu}

\author[C. Pomerance]{Carl Pomerance}
\address{Mathematics Department\\
Dartmouth College\\
Hanover, NH 03755, USA}
\email{carl.pomerance@dartmouth.edu}

\date{December 9, 2011}

\begin{abstract}
We show that there are sets of integers with asymptotic density arbitrarily close to 1
in which there is no solution to the equation $ab=c$, with $a,b,c$ in the set.
We also consider some natural generalizations, as well as a specific numerical
example of a product-free set of integers with asymptotic density greater than $1/2$.
\end{abstract}


\dedicatory
{Dedicated to Professor Andrzej Schinzel on his 75th birthday}
%\date{November 21, 2011, November 17, 2011/September 24, 2011}

\maketitle

%\tableofcontents
%***************************
%
% section 1: Introduction
%
%****************************
\section{Introduction}

We say  a set of integers $\S$ is {\em product-free} if whenever $a,b,c\in \S$ we have
$ab\ne c$.  Similarly, if $\S\subset\Z/n\Z$, we say $\S$ is product-free if
$ab\not\equiv c\pmod n$, whenever $a,b,c\in\S$.  Clearly, if $\S$ is a product-free
subset of $\Z/n\Z$, then the set of integers congruent modulo~$n$ to some member of $\S$
is a product-free set of integers.
For a positive integer $n$, let $D(n)$ denote the maximum 
value of  $|\S|/n$ where $\S$ 
runs over all  product-free
subsets of $\Z/n\Z$. (Here $|\S|$ denotes the cardinality
of a set $\S$.)

In a recent paper, the third author and Schinzel~\cite{PS}  obtained
an upper bound on $D(n)$ valid for a large set of $n$.
They showed that $D(n)<1/2$ whenever $n$ is not divisible by a square
with at least 6 distinct prime factors.  Further, those numbers
which are divisible by a square with at least 6 distinct prime
factors form a set of asymptotic density about $1.56\times10^{-8}$.
Originally they suspected that $D(n) <1/2$ might hold for all $n$.

In this paper we show that for each real number $\epsilon>0$ there is
some number $n$ with $D(n)>1-\epsilon$.  Thus, there are product-free
sets of integers with asymptotic density arbitrarily close to~1.
Stated this way, the result is best possible, since no
product-free set can have density $1$.  Indeed, if $\S$ is a
product free set of positive integers and $a$ is the least member of
$\S$, then it is easy to see that the upper density of $\S$ is at most
$1-1/(2a)$; see Remark~\ref{rem:upper}.

A consequence of our main result is that  the set of integers $n$ having $D(n) > 1-\epsilon$
has a positive lower density.
This follows using the  property that  $D(mn)\ge D(n)$
for all positive integers $m,n$. If  $D(n_0)>1-\epsilon$, 
then it shows that $D(n)>1-\epsilon$ holds for every multiple of $n_0$, and so it holds
for a set of positive integers $n$ of positive lower density.  Furthermore
the set $\NN(u)=\{ n \ge 1:  D(n) > u\}$ has a well-defined
logarithmic density $\delta(u)$ which is positive for $0 <u<1$.
In Theorem \ref{thm:main} we obtain a quantitative rate at which $D(n)$ approaches $1$,
which yields a lower bound for $\delta(u)$ as $u \to 1^-$, given as (\ref{eq:logbd}) in Sec. \ref{sec5}.

We also compute a numerical example of a number $n$ with $D(n)>1/2$
and we consider some generalizations of the equation $ab=c$.

It is interesting to note that while 
there are product-free subsets with density arbitrarily close to~1,
the density of {\em sum-free} subsets of finite abelian groups (written
additively) is easily seen to be bounded by $1/2$ (see ~\cite{GR05} for
a complete characterization of the maximum density of sum-free subsets
of various types of finite abelian groups).

%***************************
%
% SECTION 2: Main Thm
%
%***************************
\section{The main theorem}

In this section we show that there can be product-free sets of
integers of density arbitrarily close to one, but not equal to one.
Our main result is as follows.
%**********************
%  theorem 2 (formerly theorem 3}
%*********************
\begin{theorem}
\label{thm:main}
There is a positive constant $C$ and infinitely many integers $n$ with
$$
 D(n)>1-\frac{C}{(\log\log n)^{1-\frac12\e\log2}}.
$$
\end{theorem}

Here the exponent $1- \frac{1}{2} \e \log 2 \approx 0.057915$.

%**********************
%  corollary 3
%*********************

\begin{cor}
\label{cor:main}
For each real number $\epsilon>0$ there is a positive integer $n$ with 
$ D(n)>1-\epsilon$.
\end{cor}

We first sketch  the idea of the proof.
Let $\Omega(m)$ denote the number of prime factors of $m$
counted with multiplicity.  Clearly for any fixed $z$,
the set of numbers $m$ with $z<\Omega(m)<2z$ is
product-free.  Further, after Hardy and Ramanujan,
we know that $\Omega(m)$ for numbers $m\le x$ is usually
concentrated near $\log\log x$.  So if $z\approx\frac23\log\log x$
(actually $\frac\e4$ works out a little better than $\frac23$),
we have a product-free set that has the great preponderance
of integers in $[1,x]$.  With an extra device (see Lemma~\ref{lem:divisors})
for creating such a set that is periodic modulo some particular large
number $n$, we obtain the result.  The idea used bears some
resemblance to that of Remark~2 and its proof in Hajdu, Schinzel, and Skalba~\cite{HSS}.

Before giving the proof, we establish some preliminary lemmas.
Let $\varphi$ denote Euler's function 
and let $\rad(n)$ denote the largest squarefree divisor of the positive integer $n$.

%**********************
%  lemma4
%*********************

\begin{lemma}
\label{lem:divisors}
Suppose that $n$ is a positive integer and $\D$ is a product-free set of divisors of 
$n/\rad(n)$.  Then
$$
\S_\D:=\{s\in\Z/n\Z: \gcd(s,n)\in\D\}
$$
is product-free and
$$
|\S_\D|=\varphi(n)\sum_{d\in\D}\frac1d.
$$
\end{lemma}
\begin{proof}
Suppose $s_1,s_2\in\S_\D$ with $\gcd(s_i,n)=d_i\in\D$ for $i=1,2$.  We have
$\gcd(s_1s_2,n)=\gcd(d_1d_2,n)=d_3$, say.  If $d_3\nmid n/\rad(n)$, then
by hypothesis $d_3\not\in\D$, so $s_1s_2\not\in\S_\D$.  On the other hand, 
if $d_3\mid n/\rad(n)$, then $d_3=d_1d_2$, so again by hypothesis, $d_3\not\in\D$
and $s_1s_2\not\in\S_\D$.
Thus, $\S_\D$ is product-free and it remains to compute its cardinality.  For $d\in\D$,
we have
$$
\{s\in\Z/n\Z:\gcd(s,n)=d\}=\{jd:j\in\Z/(n/d)\Z,~\gcd(j,n/d)=1\}.
$$
Thus, $|\S_\D|=\sum_{d\in\D}\varphi(n/d)$.  
But, by hypothesis, we have $\rad(n/d)=\rad(n)$ for $d\in\D$, so
that $\varphi(n/d)=\varphi(n)/d$.  This completes the proof.
\end{proof}

For an integer $n>1$, let $P(n)$ denote the largest prime factor of $n$
and let $P(1)=1$.  As above, we let $\Omega(n)$ denote the number of prime factors 
of $n$, counted with multiplicity.  We use the notation $f(x)\asymp g(x)$
if there are positive constants $c_1,c_2$ such that $c_1g(x)\le f(x)\le c_2(x)$
in some stated domain for the variable $x$.
Lemma~\ref{lem:average} and Corollary~\ref{cor:hr} below are standard results,
cf.\ Exercises 04 and 05 in~\cite{HT}; we give the details for completeness.

\begin{lemma}
\label{lem:average}
Uniformly for real numbers $x,z$ with $x\ge2$ and $0<z<2$, 
$$
\sum_{P(n)\le x}\frac{z^{\Omega(n)}}{n}
\asymp\frac1{2-z}(\log x)^z.
$$ 
\end{lemma}
\begin{proof}
We have
\begin{align*}
\sum_{P(n)\le z}\frac{z^{\Omega(n)}}{n}&=\prod_{p\le x}\left(1+\frac zp+\frac{z^2}{p^2}+\cdots\right)
=\prod_{p\le x}\left(1-\frac zp\right)^{-1}\\
&=\prod_{p\le x}\left(1-\frac1p\right)^{-z}
\prod_{p\le x}\left(1-\frac1p\right)^z\left(1-\frac zp\right)^{-1}.
\end{align*}
By the theorem of Mertens we have $\prod_{p\le x}(1-1/p)^{-z}\sim\e^{\gamma z}(\log x)^z$
uniformly for $z$ in the interval $(0,2)$, as $x\to\infty$, where $\gamma$ is
the Euler--Mascheroni constant.  Thus, it suffices to prove that the second product
above is of magnitude $1/(2-z)$.  Using the power series for $\log(1-t)$, 
we have
\begin{multline*}
\log\left(\prod_{p\le x}\left(1-\frac1p\right)^z\left(1-\frac zp\right)^{-1}\right)=
 \sum_{p\le x}\left(z\log\left(1-\frac1p\right)-\log\left(1-\frac zp\right)\right)\\
=z\log\frac12-\log\left(1-\frac z2\right)+O\left(\sum_{3\le p\le x}\frac1{p^2}\right)
=-\log(2-z)+O(1).
\end{multline*}
This then completes the proof of the lemma.
\end{proof}

We will use the entropy-like function $Q(x)$ defined for $x>0$ by
$$
Q(x)=x\log x-x+1.
$$
Note that $Q(x)\ge0$ for all $x>0$ with equality only at $x=1$.

\begin{cor}
\label{cor:hr}
Uniformly for real numbers $\alpha,\beta,x$ with $0<\alpha\le 1\le\beta<2$ and $x\ge 3$,
we have
$$
\sum_{\substack{P(n)\le x\\\Omega(n)\le\alpha\log\log x}}\frac1n\ll(\log x)^{1-Q(\alpha)},
\quad
\sum_{\substack{P(n)\le x\\\Omega(n)\ge\beta\log\log x}}\frac1n\ll\frac1{2-\beta}
(\log x)^{1-Q(\beta)}.
$$
\end{cor}
\begin{proof}
We have 
\begin{align*}
\sum_{\substack{P(n)\le x\\\Omega(n)\le\alpha\log\log x}}\frac1n&\le
\sum_{P(n)\le x}\frac{\alpha^{\Omega(n)-\alpha\log\log x}}{n}\\
&=\sum_{P(n)\le x}\frac{\alpha^{\Omega(n)}}{n}(\log x)^{-\alpha\log\alpha}
\ll (\log x)^{\alpha-\alpha\log\alpha},
\end{align*}
using $0<\alpha\le1$ and
Lemma~\ref{lem:average} with $z=\alpha$.  Similarly, Lemma~\ref{lem:average} with $z=\beta$ gives
$$
\sum_{\substack{P(n)\le x\\\Omega(n)\ge\beta\log\log x}}\frac1n\le
\sum_{P(n)\le x}\frac{\beta^{\Omega(n)-\beta\log\log x}}{n}
%=\sum_{P(n)\le x}\frac{\beta^{\Omega(n)}}{n}(\log x)^{-\beta\log\beta}
\ll \frac1{2-\beta}(\log x)^{\beta-\beta\log\beta}.
$$
This completes the proof of the corollary.
\end{proof}

%**********************
% Proof of main theorem.
%*********************

\medskip\noindent
{\em Proof of Theorem} \ref{thm:main}.
Let $x$ be a large real number, let $\ell_x$ denote the least common multiple of
the integers in $[1,x]$, and let $n_x=\ell_x^2$.
Thus, by the prime number theorem, we have
$n_x=\e^{(2+o(1))x}$ as $x\to\infty$, so that
\begin{equation}
\label{eq:n_x}
\log\log n_x=\log x+O(1).
\end{equation}
Let
$$
\D_x=\left\{d\mid \ell_x\,:\,\frac \e4\log\log x<\Omega(d)<\frac \e2\log\log x\right\}.
$$
We note that each $d\in\D_x$ divides $n_x/\rad(n_x)$ and that
$\D_x$ is product-free.  Thus, by Lemma~\ref{lem:divisors} 
we find that
$$
\S_{\D_x} := \{ a \in \Z/n_x \Z:  \gcd(a, n_x)  \in \D_x\}
$$
is a product-free subset of $\Z/n_x\Z$, with density $\D(\S) = \frac{\varphi(n_x)}{n_x}\sum_{d \in \D_x} \frac{1}{d}.$
Using \eqref{eq:n_x} it
suffices to show that for some positive constant $c$ and $x$ sufficiently large,
\begin{equation}
\label{eq:toshow}
\frac{\varphi(n_x)}{n_x}\sum_{d\in\D_x}\frac1d\ge1
-\frac{c}{(\log x)^{1-\frac12\e\log2}}.
\end{equation}

We have
$$
\sum_{d\in\D_x}\frac1d\ge\sum_{d\mid\ell_x}\frac1d-
\sum_{\substack{P(d)\le x\\\Omega(d)\le\frac\e4\log\log x}}\frac1d
-\sum_{\substack{P(d)\le x\\\Omega(d)\ge\frac\e2\log\log x}}\frac1d.
$$
Since $1-Q(\frac\e4)=1-Q(\frac\e2)=\frac12\e\log2$, Corollary~\ref{cor:hr}
implies there is some absolute constant $c'>0$ with
$$
\sum_{d\in\D_x}\frac1d\ge\sum_{d\mid\ell_x}\frac1d-c'(\log x)^{\frac12\e\log2}.
$$
Now, letting $\sigma$ denote the sum-of-divisors function, 
\begin{align*}
\sum_{d\mid\ell_x}\frac1d&=\frac{\sigma(\ell_x)}{\ell_x}
=\prod_{p^a\|\ell_x}\frac{p^{a+1}-1}{p^a(p-1)}
=\prod_{p\le x}\frac p{p-1}\prod_{p^a\|\ell_x}\left(1-\frac1{p^{a+1}}\right)\\
&\ge\prod_{p\le x}\frac p{p-1}\cdot\left(1-\frac1x\right)^{\pi(x)}
\ge\prod_{p\le x}\frac p{p-1}\cdot\left(1-\frac{\pi(x)}{x}\right),
\end{align*}
where $\pi(x)$ denotes the prime-counting function.
Thus, since $\varphi(n_x)/n_x=\prod_{p\le x}(p-1)/p$,
$$
\frac{\varphi(n_x)}{n_x}\sum_{d\in\D_x}\frac1d
\ge 1-\frac{\pi(x)}{x}-c'(\log x)^{\frac12\e\log2}\prod_{p\le x}\frac{p-1}p.
$$
Using the theorem of Mertens for the product and the Chebyshev estimate
$\pi(x)\ll x/\log x$, we obtain \eqref{eq:toshow}, completing the
proof of Theorem~\ref{thm:main}.
\hfill\qedsymbol

\begin{rem}
\label{rem:stronger}
It is possible to uniformly save a factor $\sqrt{\log\log x}$ in Corollary~\ref{cor:hr}
under the strengthened hypothesis that $\alpha\in[\epsilon,1-\epsilon]$ and
$\beta\in[1+\epsilon,2-\epsilon]$, where $\epsilon>0$ is fixed but arbitrary.
This gives a slightly stronger version of Theorem~\ref{thm:main}:  There is
a positive constant $C$ such that
\begin{equation}
\label{eq:stronger}
D(n)>1-\frac{C}{(\log\log n)^{1-\frac12{\rm e}\log2}\sqrt{\log\log\log n}}
\hbox{ for infinitely many }n.
\end{equation}
The details are presented in a sequel paper~\cite{KLP}, where the principal result
is that~\eqref{eq:stronger}, apart from the constant $C$, is best possible.
\end{rem}

\begin{rem}
\label{rem:upper}
For a set $\S$ of positive integers, let $\S(x)=\S\cap[1,x]$.
If $\S$ is product-free with least member $a$, then
its upper asymptotic density, defined as
$$
\overline{d}(\S) := \limsup_{x \to \infty} \frac{1}{x} |\S(x)|,
$$
satisfies $\overline{d}(\S)  \le 1-\frac{1}{2a}$.
To see this, suppose $x \ge a$ is arbitrary.  Since $\S(x)\setminus \S(x/a)$
lies in $(x/a,x]$, we have $|\S(x)|- |\S(x/a)|\le x-\lfloor x/a\rfloor$.
Also, multiplying each member of $\S(x/a)$ by $a$ creates products in $[1,x]$ which
cannot lie in $\S$, so we have $|\S(x)|\le x- |\S(x/a)|$.  Adding these two
inequalities leads to $|\S(x)|\le x-\frac12\lfloor x/a\rfloor$, 
which proves the assertion.
\end{rem}

%**********************
%  SECTION 3
%*********************

\section{Generalizations}

If $k,j$ are positive integers, we say a set of integers (or residue classes in $\Z/n\Z$)
is $(k,j)$-product-free if there is no solution to $a_1a_2\dots a_k=b_1b_2\dots b_j$
with all $k+j$ letters being elements of the set.  If $k=j$ then only the empty
set is $(k,j)$-product-free.  Indeed, if $a$ is an element of the set, the equation
$a^k=a^k$ shows that we cannot avoid $a_1a_2\dots a_k=b_1b_2\dots b_j$.  Thus we restrict to cases where
 $k\ne j$, and we may as well assume that $k>j$.  The case of $k=2,j=1$
is the unadorned definition of product-free that was considered in the last section.
In this section we record the following simple generalization.

%********************
% Theorem 7
%*******************

\begin{theorem}
\label{thm:general}
For each real number $\epsilon>0$ and integer $m \ge 3$
there is a positive integer $n$ 
and a subset $\S$ of $\Z/n\Z$ of
cardinality at least $(1-\epsilon)n$ that is simultaneously
$(k, j)$-product-free for all positive integers $k > j$ with $k+j \le m$.
\end{theorem}
\begin{proof}
As in the proof of Theorem~\ref{thm:main}, let $\ell_x$ denote the least common
multiple of the integers in $[1,x]$, but now we let $n_x=\ell_x^m$, and
$$
\D_x=\left\{d\mid\ell_x\,:\,\left(1-\frac1m\right)\log\log x<\Omega(d)<
\left(1+\frac1m\right)\log\log x\right\}.
$$
Let $k>j$ be positive integers with $k+j\le m$.
If $d_1,\dots,d_k\in\D_x$ and also $d_1',\dots,d_j'\in\D_x$, it is easy to see that
$d=d_1\dots d_k$ and $d'=d_1'\dots d_j'$ are divisors of $n_x$.  In
addition, $d\ne d'$, since
$\Omega(d)>k(1-\frac1m)\log\log x \geq j(1+\frac1m)\log\log x>\Omega(d')$.  
Thus,
$\D_x$ is $(k,j)$-product-free as is the set $\S_{\D_x}$ (cf.\ Lemma~\ref{lem:divisors}).  
As in the proof of
Theorem~\ref{thm:main} it suffices to show that for each $\epsilon>0$,
$$
\frac{\varphi(n_x)}{n_x}\sum_{d\in\D_x}\frac1d\ge1-\epsilon
$$
for all sufficiently large $x$ depending on $\epsilon$.  Already from the proof of
Theorem~\ref{thm:main}, we have
$$
\frac{\varphi(n_x)}{n_x}\sum_{d\mid\ell_x}\frac 1d\ge1-\frac{\pi(x)}{x}\sim1
$$
as $x\to\infty$.
Since $\varphi(n_x)/n_x\sim1/(e^\gamma\log x)$ as $x\to\infty$,
it suffices to show that
\begin{equation}
\label{eq:suff}
\sum_{\substack{d\mid\ell_x\\d\not\in\D_x}}\frac1d=o(\log x)~\hbox{ as }x\to\infty.
\end{equation}
Letting $\delta_1=Q(1-1/m)$ and $\delta_2=Q(1+1/m)$, we have
$\delta_1,\delta_2>0$.  Using Corollary~\ref{cor:hr},
$$
\sum_{\substack{d\mid \ell_x\\ \Omega(d)\le\left(1-\frac1m\right)\log\log x}}\frac1d
\le(\log x)^{1-\delta_1/2},\quad
\sum_{\substack{d\mid \ell_x\\ \Omega(d)\ge\left(1+\frac1m\right)\log\log x}}\frac1d
\le(\log x)^{1-\delta_2/2}
$$
for all large $x$.  Thus, we have \eqref{eq:suff}, which completes the proof of the theorem.
\end{proof}

Returning to the case when $k=j$, 
we can redefine the notion of $(k,k)$-product-free to mean
that the equation $a_1a_2\dots a_k=b_1b_2\dots b_k$ 
implies that $\{a_1,a_2,\dots,a_k\}=\{b_1,b_2,\dots,b_k\}$ as multisets.
For example, the primes are $(k,k)$-product-free for every $k$.  This is essentially
a best-possible result, for as shown by Erd\H os~\cite{E38} in 1938, if $\S$ is a subset of
the positive integers which is $(2,2)$-product-free, then the number of members
of $\S$ in $[1,x]$ is $\pi(x)+O(x^{3/4})$.

The equation
$abc=d^2$ was recently considered in~\cite{HSS}, where it was shown (see Corollary~1) that
if $\S$ is a set of integers such that 
$$
abc=d^2~~\,\,\, \mbox{has no solution with} \,~~ a, b, c \in \S,  \,\,\, d~~\mbox{arbitrary},
$$
then the lower asymptotic density of $\S$ is at most $1/2$.
This result was inadvertently misquoted in~\cite{PS}, where it was asserted that such
a result holds with all of $a,b,c,d\in\S$.  In fact, this is false since
Theorem \ref{thm:general}
applied with $(k, j)=(3, 2)$ implies the complementary result that for any
$\epsilon >0$ there exists a set $\S$ of density exceeding  $1- \epsilon$ such that 
\begin{equation}\label{eq:S2}
abc=d^2~~\,\,\,\mbox{has no solution with} \,\,\,~ a, b, c , d\in \S.
\end{equation}
More precisely, it gives:

%********************
% Corollary 8
%*******************

\begin{cor}
\label{cor:general}
For each real number $\epsilon>0$,
there is a positive integer $n$ and a  subset $\S$ of $\Z/n\Z$ of
cardinality at least $(1-\epsilon)n$  such that $abc=d^2$
has no solution with $a, b, c, d \in \S$.
\end{cor}

%**********************
%  SECTION 4
%*********************

\section{A numerical example}

In this section we give the details for 
a number $N$ for which there exists a product-free subset of $\Z/N\Z$
of size larger than $N/2$.  Our example is very large; it would be of
interest to see if a substantially smaller number could be found.

Let $\PP$ denote the set of the first $10{,}000{,}000$ primes and let $Q$ be their
product.  For each positive
integer $j$, let
$$
\sigma_j=\sum_{p\in\PP}\frac{1}{p^j},\quad 
S_j=\sum_{\substack{\rad(m)\mid Q\\ \Omega(m)=j}}\frac1m.
$$
We have computed these sums for $j$ up to 13, finding that to 6
decimal places,
$$
\begin{array}{rrrr}
\sigma_1=3.206219,&\sigma_2=0.452247,&\sigma_3=0.174763,&\sigma_4=0.076993,\\
\sigma_5=0.035755,&\sigma_6=0.017070,&\sigma_7=0.008284,&\sigma_8=0.004061,\\
\sigma_9=0.002004,&\sigma_{10}=0.000994,&\sigma_{11}=0.000494,&\sigma_{12}=0.000246,\\
\sigma_{13}=0.000123~&&&
\end{array}
$$
and
$$
\begin{array}{rrrr}
S_1=3.206219,&S_2=5.366043,&S_3=6.276492,&S_4=5.796977,\\
S_5=4.529060,&S_6=3.130763,&S_7=1.976769,&S_8=1.167289,\\
S_9=0.656256,&S_{10}=0.356061,&S_{11}=0.188345,&S_{12}=0.097866,\\
S_{13}=0.050226.&&&
\end{array}
$$
Concerning these calculations, we note that the computation for $\sigma_1=S_1$
is the most time consuming.  The other values of $\sigma_j$ represent the starts
of rapidly converging series, and in fact these values can be found on the web
as values of the ``prime zeta function.''  The remaining values of $S_j$ are
easily computed by a hand calculator using the identity
$$
S_k=\frac1k\sum_{j=1}^k\sigma_jS_{k-j},
$$
where by convention we take $S_0=1$ (see~\cite[page 23, (2.11)]{Mac}).

Let 
$$
N=Q^{14}=\prod_{p\in\PP}p^{14}
$$
and let
$$
\D=\{d\mid N\,:\, 3\le \Omega(d)\le5~\hbox{ or }~11\le\Omega(d)\le13\}.
$$
A moment's reflection shows that $\D$ is product-free and that each member of $\D$ divides
$N/\rad(N)$, and so from Lemma~\ref{lem:divisors},
$$
\S_\D=\{m\bmod N:\gcd(m,N)\in \D\}
$$
is also product-free.
Further,
\begin{equation}
\label{eq:Ssize}
\frac{|\S_\D|}{N}=\frac{\varphi(N)}{N}\sum_{d\in\D}\frac1d.
\end{equation}
We may compute $\varphi(N)/N$ using $\sigma_1$ and $\sigma_2$ as follows:
$$
\log\frac{\varphi(N)}{N} =\sum_{p\in\PP}\log\Big(1-\frac1p\Big)= -\sigma_1-\frac12\sigma_2
+\sum_{p\in\PP}\left(\frac1p+\frac1{2p^2}+\log\Big(1-\frac1p\Big)\right).
$$
The remaining sum above is the start of a rapidly converging series,
so we easily find that
\begin{equation}
\label{eq:phiest}
\frac{\varphi(N)}{N}>0.029542.
\end{equation}
The sum in \eqref{eq:Ssize} is
$$
\sum_{d\in\D}\frac1d=S_3+S_4+S_5+S_{11}+S_{12}+S_{13} =16.938967.
$$
Thus, with \eqref{eq:Ssize} and \eqref{eq:phiest}, we have
$$
\frac{|\S_\D|}{N}>(0.029542)(16.9389)>0.5004.
$$

This number $N$ is very large, it is about $10^{1.09\times10^9}$.
However, it is 
possible to reduce the exponents somewhat for the larger primes in $N$.  Let $N'$ be $N$
divided by the 12th power of each prime dividing $N$ that is above $10^6$.
Then
$D(N')>0.5003 N'$ and $N'$ is about $10^{1.61\times10^8}$.  
We have made some effort
at finding a smaller example, say below $10^{10^8}$, but we were not successful.
%**********************
%  SECTION 5
%*********************

\section{Densities and further problems}\label{sec5}

Let $u\in[0,1)$ be a real number and, as in the introduction, let
$\NN(u)$ denote the set of natural numbers $n$ with $D(n)>u$.
  Since $D(mn)\ge D(n)$,
it follows that if $n\in\NN(u)$, so too is every multiple of $n$. 
Consequently
$\NN(u)$ has a logarithmic density, see~\cite{DE1,DE2}, denote this by $\delta(u)$.
We have by Corollary~\ref{cor:main} that $\delta(u)>0$ for all $u\in[0,1)$.  We can
say a bit more.

%**********************
% Proposition 9
%*********************
\begin{prop}
\label{prop7}
We have $\liminf_{n\to\infty}D(n)=1/2$. Consequently for
$0 \le u < \frac{1}{2}$ the set $\NN(u)$ has both a  logarithmic
density $\delta(u)$  and a natural density $d(u)$ satisfying 
$d(u)= \delta(u)=1$.
\end{prop}
\begin{proof}
Let $p$ be an odd prime and let $a$ be a positive integer.  The set of nonzero
 residues mod~$p^a$
which are the product of a power of $p$ and a quadratic nonresidue mod~$p$ is product-free,
and this shows
that $D(p^a)\to\frac12$ as $a\to\infty$ (recall that
$D(n)<1/2$ if $n/\rad(n)$ does not have at least 6 distinct prime factors).
%$\omega(n/\rad(n))\leq5$).  
In addition, the set of nonzero 
residues mod $2^a$ which are the product of a power of 2 and an integer that is 3~mod~4 is
product-free, so that $D(2^a)\to\frac12$ as $a\to\infty$.  Since $D(p)\to\frac12$ as $p\to\infty$
through the primes, it follows that $D(q)\to\frac12$ as $q\to\infty$ through the prime powers
(which include the primes).
Hence for each real number $\epsilon>0$, there are at most finitely many
prime powers $q$ with $D(q)\le\frac12-\epsilon$.  Thus, if $D(n)\le \frac12-\epsilon$, it
follows that each prime power dividing $n$ must come from this set, forcing the set of such
$n$ to be finite as well.  This proves the first statement in the proposition.  
Let $u\in[0,1/2)$.  By what we just proved, the set $\NN(u)$ consists of
all but finitely many natural numbers.  This establishes 
the second statement in the proposition.
\end{proof}

It follows from the principal results of \cite{PS} that $\delta(1/2)\le1.56\times10^{-8}$,
and so with Proposition~\ref{prop7} it follows that $\delta(u)$ is not continuous in the
variable $u$ at $1/2$.  From the numerical example in the last section, we have 
$\delta(1/2)>10^{-1.62\times10^8}$.  
There is of course an enormous (multiplicative) gap between these
two bounds for $\delta(1/2)$.

More generally Theorem \ref{thm:main} yields a lower bound for 
$\delta(u)$ as $u \to 1^-$.  Setting 
$\alpha_0 := (1- \frac{1}{2}\e \log 2)^{-1} \approx 17.26659$,
we have
\begin{equation}
\label{eq:logbd}
\delta(u) > 1/\exp\exp\left((C/(1-u))^{\alpha_0}\right).
\end{equation}
Note that \eqref{eq:stronger} allows a slight improvement in this estimate.

It seems likely that for each $u$, the set $\NN(u)$ has an asymptotic density
$d(\NN(u))$. General facts about  asymptotic  densities give 
$\underline{d}(\NN(u)) \le \delta(u) \le \overline{d}(\NN(u)),$
and a natural density $d(u) = \delta(u)$ exists 
for those values with $\underline{d}(\NN(u))= \overline{d}(\NN(u)).$
Our proofs show that $\underline{d}(\NN(u)) >0$ for $0 < u < 1$
and $ \overline{d}(\NN(u))< 1$ for $u \ge \frac{1}{2}.$

As asked in \cite{PS},
is it true that for $u\ge1/2$, the ``primitive'' members of $\NN(u)$ (namely, they
are not divisible by any other member of $\NN(u)$) are all squarefull?  If so,
then it would follow that the asymptotic density of $\NN(u)$ exists for each value of $u$.

\subsection*{Acknowledgments}
We thank Rosa Orellana for a helpful discussion concerning~\cite{Mac}.
Part of this work was done while the three authors visited MSRI,
as part of the semester program Arithmetic Statistics.
They thank MSRI for support, funded through the NSF.
The first author was supported in part by grants from
the G\"oran Gustafsson Foundation,
  the Knut and Alice Wallenberg foundation, and the Swedish Research Council.
The second author was supported in
part by NSF grant DMS-0801029.  The third author was supported in
part by NSF grant DMS-1001180.

\end{document}